\documentclass[a4paper,12pt,leqno]{amsart}
\usepackage[a4paper, top=3cm, bottom=3cm, left=3cm, right=3cm]{geometry}
\usepackage{amsmath,amsfonts,amsthm,amssymb,dsfont}
\usepackage{amsrefs}
\usepackage[OT4]{fontenc}
\usepackage{enumerate}
\usepackage{mathrsfs}
\usepackage{enumerate, xspace}
\usepackage{graphicx}
\usepackage{color}
\usepackage{hyperref}

\long\def\symbolfootnote[#1]#2{\begingroup
	\def\thefootnote{\fnsymbol{footnote}}\footnote[#1]{#2}\endgroup}

\newtheorem{theorem}{Theorem}[section]
\newtheorem{lemma}[theorem]{Lemma}
\newtheorem{thm}[theorem]{Theorem}

\newtheorem{prop}[theorem]{Proposition}

\theoremstyle{definition}
\newtheorem{rem}[theorem]{Remark}
\newtheorem{defin}[theorem]{Definition}

\newcommand{\N}{\mathbb{N}}

\newcommand{\W}{\mathcal{W}}
\newcommand{\A}{\mathcal{A}}
\newcommand{\V}{\mathcal{V}}
\newcommand{\LL}{\mathcal{L}}
\newcommand{\id}{\text{id}}
\newcommand{\vr}{\nu}

\begin{document}
	
	\title{Garside shadows and biautomatic structures in Coxeter groups}
	
	\author[F.~Dos Santos]{Fabricio Dos Santos}

    \address{
		Department of Mathematics and Statistics,
		McGill University,
		Burnside Hall,
		805 Sherbrooke Street West,
		Montreal, QC,
		H3A 0B9, Canada}
	\email{fabricio.dossantos@mail.mcgill.ca}
	
	\maketitle
	
	\begin{abstract}
		\noindent In 2022, Osajda and Przytycki showed that any Coxeter group $W$ is biautomatic. Key to their proof is the notion of voracious projection of an element $g \in W$, which is used iteratively to construct a biautomatic structure for~$W$: the voracious language. In this article, we generalize these two notions by defining them for any Garside shadow $B$ in a Coxeter system $(W,S)$. This leads to the result that any finite Garside shadow in $(W,S)$ can be used to construct a biautomatic structure for $W$. In addition, we show that for the Garside shadow~$L$ of low elements, the biautomatic structure obtained corresponds to the original voracious language of Osajda and Przytycki. These results answer a question of Hohlweg and Parkinson. 
	\end{abstract}

\section{Introduction}

Coxeter groups are widely studied in mathematics as abstractions of reflection groups. They were introduced by H.S.M.\ Coxeter in 1934 \cite{Coxeter-1934}, and since then many of their important algebraic, geometric and algorithmic properties have been established. One of these properties is biautomaticity, which was introduced in the book \cite{wordprocessing} and roughly says that there is a ``well behaved'' system of paths in the Cayley graph of the group (see Section \ref{section Biau} for a rigorous definition). In particular, biautomaticity is strongly linked to algorithmic properties of a group and provides a good setting for studying its structure. 

A slightly weaker version of this property was established for Coxeter groups in the 1990s. Namely, Davis and Shapiro proved, under the assumption of the Parallel Wall Theorem, that Coxeter groups are automatic \cite{DS-au}. The result was then completed by Brink and Howlett with their proof of the Parallel Wall Theorem \cite{BH-par}. The problem of proving that Coxeter groups are biautomatic, however, remained open for much longer. Partial results were obtained in multiple subclasses of Coxeter groups, namely Euclidean and Gromov hyperbolic \cite{wordprocessing}, right-angled \cite{NR-bi, NR-cox}, no Euclidean reflection triangles \cite{Ba-bi, CM-bi}, relatively hyperbolic \cite{Ca-bi} and 2-dimensional \cite{MOP-bi}. Finally, in 2022, Osajda and Przytycki proved the result for a general Coxeter group $W$ \cite{OP-bi}.

Crucial to the proof of biautomaticity by Osajda and Przytycki is the notion of \emph{voracious projection} $p(g)$ of an element $g \in W$, which is used iteratively to construct a biautomatic structure for $W$, namely the \emph{voracious language} $\V$ \cite{OP-bi}. In this paper, we study and generalize these two notions, leading to a new class of biautomatic structures for the Coxeter group $W$. In particular, we define a voracious projection $\vr_B$ and voracious language $\V_B$ for any Garside shadow $B$ in a Coxeter system $(W,S)$, that is, a subset $B \subseteq W$ containing $S$ which is closed under taking suffixes and joins (in the right weak order), introduced in \cite{DH-Garside}. These definitions are motivated by an observation of Hohlweg and Parkinson, who reinterpreted the voracious projection $p(g)$ as the $g$-translate of a particular element of the set $L$ of \emph{low elements}, which is a finite Garside shadow in $(W,S)$. The main result of this paper is the following, answering in the positive a question of Hohlweg and Parkinson.  

\begin{thm} \label{mainthm}
Let $B$ be a finite Garside shadow in $(W,S)$. Then $(S,\V_B)$ is a biautomatic structure for $W$.
\end{thm}

We also show in Proposition \ref{vorOG} that the original voracious projection $p$ and voracious language $\V$ of Osajda and Przytycki are obtained from $\vr_B$ and $\V_B$ as a special case when $B = L$. 

\medskip

\noindent \textbf{Organization.} We start by recalling in Section \ref{prelim} definitions related to Coxeter groups, walls and the right weak order. In Section \ref{section Garside}, we define Garside shadows and their associated projections, and give examples and properties relevant to our result. Then, in Section \ref{section Biau}, we define biautomatic structures, and introduce the notion of voracious projection $\vr_B$ and voracious language $\V_B$ of a Garside shadow $B$. We proceed by defining regular languages in Section \ref{section regularity} and showing that $\V_B$ is regular when $B$ is finite. Finally, we complete the proof of Theorem \ref{mainthm} in Section \ref{section FTP} by proving the fellow traveller property for $\V_B$ when $B$ is finite.  

\medskip

\noindent \textbf{Acknowledgments.} We express our sincere gratitude to Piotr Przytycki for his invaluable guidance, support, and numerous suggestions throughout the development of this work. We also thank Christophe Hohlweg for enlightening discussions, and the anonymous referee for useful comments on a first version of this article.

\section{Preliminaries} \label{prelim}

\subsection{Coxeter groups and words} \label{Coxeter groups and words}

Let $S$ be a finite set and $W$ be a group generated by $S$ subject only to relations $s^2=\id$ for $s \in S$ and $(st)^{m_{st}} = \id$ for $s \neq t \in S$, where $m_{st} = m_{ts} \in \{2,3,\ldots,\infty\}$, and $m_{st} = \infty$ means that we do not impose a relation between $s$ and $t$. Such a pair $(W,S)$ is called a \emph{Coxeter system} and $W$ is called a \emph{Coxeter group}. The \emph{word length} of an element $g \in W$, denoted $\ell(g)$, is the minimal number $n\geq 0$ such that we can write $g = s_1\ldots s_n$ where each $s_i \in S$ for $i = 1, \ldots, n$. The \emph{word metric} on $W$ is the function $d : W \times W \to \N$ defined by $d(g,h) = \ell(g^{-1}h)$. For $g,w \in W$, we say that $w$ is a \emph{suffix} of $g$ if there exists $u \in W$ such that $g=uw$ and $\ell(g) = \ell(u) + \ell(w)$.

Denote by $S^*$ the set of all words over $S$. For $v \in S^*$ a word of length $n$, write $v(i)$ for the prefix of $v$ (that is, initial subword of $v$) of length $1 \leq i \leq n-1$ and $v$ itself for $i \geq n$. For $1 \leq i \leq j \leq n$ let $v(i,j)$ denote the subword of $v(j)$ obtained by removing the prefix $v(i-1)$. Finally, for a word $v \in S^*$ let $\overline{v} \in W$ be the element of the group $W$ represented by $v$. 

\subsection{Reflections and walls} \label{Reflections and walls}

Let $X^1$ be the \emph{Cayley graph} of $W$ with respect to $S$, that is the graph with vertex set $X^0 = W$ and edges of length 1 joining each $g \in W$ to $gs$ for $s \in S$. We consider the action of $W$ on $X^0$ by left-multiplication, which induces an action of $W$ on $X^1$. A \emph{reflection} of $W$ is an element of the set $R = \bigcup_{g\in W}gSg^{-1}$. Given $r \in R$, the \emph{wall} $\W_r$ of $r$ is the fixed point set of $r$ on $X^1$. In particular, for a wall $\W_r$ in $X^1$ such $r \in R$ is unique. Every wall $\W$ separates $X^1$ into two components, called \emph{half-spaces}. We say that two walls $\W$ and $\W'$ \emph{intersect} if $\W'$ is not contained in a half-space of $\W$ (this relation is symmetric). Furthermore, we say a wall $\W$ \emph{separates} an element $g \in W$ (or a wall $\W'$) from another element $h \in W$ (or wall $\W''$) if $g$ (or $\W$) belongs to a different half-space of $\W$ than $h$ (or $\W''$). It is proven in \cite[Lem 2.5]{R-buildings} that a geodesic edge-path in $X^1$ intersects $\W$ at most once, thus, for $g,h \in W$, we have that $d(g,h)$ equals the number of walls in $X^1$ separating $g$ and $h$. 

\subsection{The right weak order} \label{The right weak order}

We will be considering the \emph{right weak order} $\preceq$ on $W$. This partial order is defined by setting $g \preceq h$ for $g,h \in W$ if $g$ lies on a geodesic in $X^1$ from $\id$ to $h$. Equivalently, $g \preceq h$ if there is no wall in $X^1$ separating $g$ from both $\id$ and $h$. It is known that $(W, \preceq)$ is a complete meet-semilattice, that is, any subset $A$ of $W$ has a greatest lower bound $\bigwedge A$ called the \emph{meet} of $A$ (see \cite[Thm 3.2.1]{B-meet}). This implies that any bounded subset $A$ of $W$ also has a least upper bound $\bigvee A$ called the \emph{join} of $A$.

\section{Garside shadows} \label{section Garside}

\subsection{Garside shadows and the $B$-projection}

Garside shadows were introduced by Dyer and Hohlweg in \cite{DH-Garside} as analogs to the notion of Garside families in a monoid (see \cite{D-monoid, DDH-monoid}). They were used in \cite{HNW-Gar} to construct automata recognizing the language of reduced words of $(W,S)$.

\begin{defin}
A \emph{Garside shadow} in $(W,S)$ is a subset $B \subseteq W$ containing $S$ such that:
\begin{enumerate}
    \item \label{join-closed} for any $X \subseteq B$, if $\bigvee X $ exists then $\bigvee X \in B$;
    \item \label{suffix-closed} for any $b \in B$, if $w$ is a suffix of $b$ then $w \in B$.
\end{enumerate}
Condition (\ref{join-closed}) means that $B$ is closed under taking joins (when they exist), and (\ref{suffix-closed}) that $B$ is closed under taking suffixes.
\end{defin}

Important examples of finite Garside shadows are the \emph{low elements} $L$ and $m$\emph{-low elements} $L_m$, which we describe in \S\ref{subsection Shi}, and the set $\Gamma$ of \emph{gates} of the \emph{cone type partition}, described in \S\ref{subsection cone types}. Furthermore, note that the intersection of Garside shadows is a Garside shadow, so there exists a smallest Garside shadow, which turns out to be $\Gamma$.

Garside shadows come with a natural projection, introduced in \cite{HNW-Gar}, which can be defined as follows.
\begin{defin}
    Let $B$ be a Garside shadow. The map $\pi_B : W \to B$ given by
    \begin{align*}
        \pi_B(g) = \bigvee\{b \in B : b \preceq g\}
    \end{align*}
    is called the $B$\emph{-projection}.
\end{defin}

Note that this map is well-defined since the set $\{b \in B : b \preceq g\}$ is bounded, so its join always exists. Furthermore, this map satisfies $\pi_B \circ \pi_B = \pi_B$, so it is indeed a projection in this sense. It is also clear from the definition that $\pi_B(g) \preceq g$ for every $g \in W$.

The $B$-projection induces a partition $\mathscr{P}_B$ of $W$ with parts $P_b = \{g \in W : \pi_B(g) = b\}$ for $b \in B$, which is one example of a ``regular partition'' \cite[Def 3.10]{PaY-reg}. Furthermore, by \cite[Lem 4.17]{PaY-reg}, the element $b \in B$ is a \emph{gate} of the part $P_b$, meaning that $b$ is the smallest element of $P_b$ with respect to the right weak order $\preceq$. 

We now prove the following simple lemma which will be useful later on.

\begin{lemma} \label{refinement}
    Let $B$ and $B'$ be Garside shadows such that $B \subseteq B'$, and let $x,y \in W$. If $\pi_{B'}(x) = \pi_{B'}(y)$, then $\pi_B(x) = \pi_B(y)$. 
\end{lemma}

\begin{proof}
    Let $\pi_{B'}(x) = \pi_{B'}(y) = b' \in B'$, $\pi_B(x) = b_1 \in B$ and $\pi_B(y) = b_2 \in B$. Since $b_1,b_2 \in B'$, we have that $b_1,b_2 \preceq b'$ and thus their join $b = b_1 \vee b_2$ exists and belongs to $B$. Then, $b_1 \preceq b \preceq b' \preceq x$ and $b_2 \preceq b \preceq b' \preceq y$, therefore we must have that $b = b_1 = b_2$.
\end{proof}

\subsection{Shi partitions, elementary walls and low elements} \label{subsection Shi}

Let $\W$ be a wall in $X^1$, $g \in W$ and $m \in \N$. We say that $\W$ is $m$\emph{-close} to $g$ if there are at most $m$ walls separating $g$ from $\W$. The wall $\W$ is called $m$\emph{-elementary} if it is $m$-close to $\id$. We denote by $\Sigma_m$ the set of $m$-elementary walls, and note that these correspond to the $m$\emph{-small roots} in \cite{DHFM24} (also called $m$\emph{-elementary roots} in \cite{BH-par, Fu-E_m}). For $m=0$, we simply call such walls \emph{elementary} and denote by $\Sigma$ the set $\Sigma_0$. The Parallel Wall Theorem of Brink and Howlett, which was crucial to the proof that every Coxeter group is automatic, states that the set $\Sigma$ is finite \cite{BH-par}. More generally, it is also the case that $\Sigma_m$ is finite for every $m \in \N$, as proven by Fu in \cite{Fu-E_m}. 

The connected components $Y$ of $X^1 \setminus \Sigma_m$ are called $m$\emph{-Shi components} and the $m$\emph{-Shi part} associated to $Y$ is $Y \cap X^0$. The $m$\emph{-Shi partition} $\mathscr{S}_m$ is the partition of $W$ which has as parts the $m$-Shi parts. For $m=0$, we simply call the parts $Y \cap X^0$ \emph{Shi parts} and denote by $\mathscr{S}$ the partition $\mathscr{S}_0$, called the \emph{Shi partition}. The partitions $\mathscr{S}$ and $\mathscr{S}_m$ are examples of regular partitions \cite[Thm 3.11]{PaY-reg}. 

It is proven in \cite{DHFM24} in the general case, and independently in \cite{PrY-pair} for the case $m=0$, that each $m$-Shi part has a smallest element with respect to the order $\preceq$, called a \emph{gate} in \cite{PaY-reg}. Denote by $L_m$ the set of gates of the $m$-Shi partition, that is the set of smallest elements for each $m$-Shi part. Similarly, denote by $L$ the set $L_0$. Note that by \cite[Thm 1.1]{DHFM24}, $L_m$ is equal to the set of so called $m$\emph{-low elements}, and thus forms a finite Garside shadow for each $m \in \N$. This fact was proven in the case $m=0$ in \cite{DH-Garside} and independently in \cite{PrY-pair}, and the general case was completed in \cite{D-mlow}. This leads to the following result.

\begin{lemma} \label{lem-mShi}
    If $g \in W$ and $m \in \N$ then $\pi_{L_m}(g)$ is the gate of the $m$-Shi part containing $g$. Therefore, the partition $\mathscr{P}_{L_m}$ is equal to $\mathscr{S}_m$ and the partition $\mathscr{P}_L$ is equal to $\mathscr{S}$.
\end{lemma}

\begin{proof}
    Let $\pi_{\mathscr{S}_m} : W \to L_m$ be the map assigning $g \in W$ to the gate of the $m$-Shi part containing $g$. Suppose for a contradiction that $\pi_{\mathscr{S}_m}(g) \neq \pi_{L_m}(g)$, meaning that there is an $m$-elementary wall $\W$ separating $\pi_{\mathscr{S}_m}(g)$ from $\pi_{L_m}(g)$. Then, by definition of $\pi_{L_m}$, we have that $\pi_{\mathscr{S}_m}(g) \preceq \pi_{L_m}(g) \preceq g$. This implies that $\W$ also separates $\pi_{\mathscr{S}_m}(g)$ from $g$, contradicting the fact that they belong to the same $m$-Shi part. Hence, $\pi_{\mathscr{S}_m} = \pi_{L_m}$, and since two elements $x,y \in W$ belong to the same $m$-Shi part if and only if $\pi_{\mathscr{S}_m}(x) = \pi_{\mathscr{S}_m}(y)$, it follows that $\mathscr{S}_m = \mathscr{P}_{L_m}$ and $\mathscr{S} = \mathscr{P}_L$.
\end{proof}

\subsection{Cone types} \label{subsection cone types}

The \emph{cone types} in the group $W$ give rise to a second important example of a finite Garside shadow in $(W,S)$, which turns out to be the smallest one. 

\begin{defin}
    Let $g, h \in W$. The \emph{cone type of} $g$ \emph{based at} $h$ is the set
    \begin{align*}
        T_{g,h} = \{f \in W : \ell(g^{-1}h) + \ell(h^{-1}f) = \ell(g^{-1}f)\}.
    \end{align*}
    Equivalently, $T_{g,h}$ is the set of elements $f \in W$ for which there exists a geodesic in $X^1$ from $g$ to $f$ passing through $h$. The set $T_{g^{-1}, \id}$ is called the \emph{cone type of} $g$, and corresponds to the set $T(g)$ in the notation of \cite{PaY-reg}. The \emph{cone type part of} $g$ is the set of all elements $h \in W$ such that $T_{h, \id} = T_{g, \id}$. 
\end{defin}

It turns out that the set of all cone types in $W$ is finite, which is a consequence of the fact that $\Sigma$ is finite \cite[Cor 2.10]{PaY-reg}. Let $\mathscr{T}$ be the \emph{cone type partition}, that is the partition of $W$ which has as parts the cone type parts. Parkinson and Yau show in \cite{PaY-reg} that $\mathscr{T}$ is another example of a regular partition, and in particular that $\mathscr{S}$ and $\mathscr{S}_m$ are refinements of $\mathscr{T}$ (in the sense that every part of $\mathscr{S}$ or $\mathscr{S}_m$ is contained in some part of $\mathscr{T}$). The partition $\mathscr{T}$ also has the property that it is \emph{convex}, meaning that for all pairs of elements $x,y$ in a cone type part $Q$ any geodesic from $x$ to $y$ has all vertices in $Q$. This was proven in \cite[Prop 4.3]{PaY-reg} and alternatively in \cite[Rem 3.1]{PrY-pair}, and we state a special case of this as a Lemma for later use.

\begin{lemma} \label{lem-convex}
    Let $f,g,h \in W$ be such that $f \preceq g \preceq h$. If $f$ and $h$ belong to the same cone type part then $T_{f, \id} = T_{g, \id} = T_{h, \id}$.
\end{lemma}

It is also the case that every cone type part has a gate. This was proven in \cite[Thm 1]{PaY-reg} and alternatively in \cite{PrY-pair}. Let $\Gamma$ be the set of gates of $\mathscr{T}$, which is finite. The set $\Gamma$ was used in \cite{PaY-reg} to construct the minimal automaton for the language of reduced words of $(W,S)$. Furthermore, Parkinson and Yau showed in that paper that $\Gamma$ is closed under taking suffixes, and conjectured that $\Gamma$ is the smallest Garside shadow in $(W,S)$. Przytycki and Yau then verified in \cite{PrY-pair} that $\Gamma$ is also closed under taking joins, thus settling this conjecture and a conjecture of Hohlweg, Nadeau and Williams in \cite{HNW-Gar} that the automaton constructed from the smallest Garside shadow is the minimal automaton for the language of reduced words of $(W,S)$. This leads to the following result, similar in nature to Lemma \ref{lem-mShi}.

\begin{lemma} \label{lem-cone}
    If $g \in W$ then $\pi_{\Gamma}(g)$ is the gate of the cone type part containing $g$. Therefore, the partition $\mathscr{P}_{\Gamma}$ is equal to $\mathscr{T}$.
\end{lemma}

\begin{proof}
    Let $\pi_{\mathscr{T}} : W \to \Gamma$ be the map assigning $g \in W$ to the gate of the cone type part containing $g$. Suppose for a contradiction that $\pi_{\mathscr{T}}(g) \neq \pi_{\Gamma}(g)$, meaning that $T_{\pi_{\mathscr{T}}(g), \id} \neq T_{\pi_{\Gamma}(g), \id}$. Then, by definition of $\pi_{\Gamma}$, we have that $\pi_{\mathscr{T}}(g) \preceq \pi_{\Gamma}(g) \preceq g$. But this implies by Lemma \ref{lem-convex} that $T_{\pi_{\mathscr{T}}(g), \id} = T_{\pi_{\Gamma}(g), \id} = T_{g,\id}$, which is a contradiction. Hence, $\pi_{\mathscr{T}} = \pi_{\Gamma}$, and since two elements $x,y \in W$ belong to the same cone type part if and only if $\pi_{\mathscr{T}}(x) = \pi_{\mathscr{T}}(y)$, it follows that $\mathscr{T} = \mathscr{P}_{\Gamma}$. 
\end{proof}

Another important result about the Garside shadow $\Gamma$ which we will use later on is the following Proposition, verified in \cite[Prop 4.29]{PaY-reg}.

\begin{prop} \label{cone type garside}
    If $B$ is a Garside shadow then $\Gamma \subseteq B$. 
\end{prop}

\section{Biautomatic structures} \label{section Biau}

\subsection{Biautomaticity}

Biautomatic structures were introduced in the book \cite{wordprocessing}. We provide the following definition, which is equivalent to the characterisation of biautomaticity in \cite[Lem 2.5.5]{wordprocessing}, and refer the reader to \S\ref{Coxeter groups and words} for a reminder on the notations used below. 

\begin{defin}
    Let $G$ be a group. A \emph{biautomatic structure for} $G$ is a pair $(S,\LL)$ where $S$ is a finite generating set for $G$ closed under inversion and $\LL$ is a regular language (see Section \ref{section regularity} for the definition of a regular language) for which there exists constants $C$ and $C'$, satisfying the following conditions:
    \begin{enumerate}
        \item \label{surj L} for every $g \in G$, there is a word in $\LL$ representing $g$;
        \item \label{defFTP1} for every $v,v' \in \LL$ representing elements $g, g' \in G$ with $g'=gs$ for some $s \in S$, we have $d(\overline{v(i)},\overline{v'(i)}) \leq C$ for all $i\geq 1$;
        \item \label{defFTP2} for every $v,v' \in \LL$ representing elements $g, g' \in G$ with $g'=sg$ for some $s \in S$, we have $d(s\overline{v(i)},\overline{v'(i)}) \leq C'$ for all $i\geq 1$.
    \end{enumerate}
    Conditions (\ref{defFTP1}) and (\ref{defFTP2}) are called the \emph{fellow traveller property}. In particular, we will refer to condition (\ref{defFTP1}) as the first fellow traveller property, and condition (\ref{defFTP2}) as the second fellow traveller property. If $\LL$ satisfies conditions (\ref{surj L}) and (\ref{defFTP1}) but not necessarily (\ref{defFTP2}), then $(S,\LL)$ is called an \emph{automatic structure}. $G$ is said to be \emph{automatic} if it admits an automatic structure, and \emph{biautomatic} if it admits a biautomatic structure.
\end{defin}

We note that this definition is equivalent to the original definition of biautomaticity if in condition (\ref{surj L}) the set of words in $\LL$ representing each element of $G$ is finite \cite[Thm 6]{biautomaticdef}.

\subsection{The voracious projection and voracious language}

Crucial to the proof of the biautomaticity of Coxeter groups by Osajda and Przytycki \cite{OP-bi} is the notion of the \emph{voracious projection} of an element $g \in W$, which is introduced in that same paper. In order to define it, they consider the set $\W(g)$ of walls in $X^1$ that separate $g$ from $\id$ and which are closest to $g$ (that is $0$-close to $g$). Then, they define the set $P(g)$ of elements $p \in W$ such that $p \preceq g$ and $p$ is not separated from $\id$ by any wall in $\W(g)$. Finally, they show that $P(g)$ has a largest element $p(g)$ with respect to $\preceq$ and set the voracious projection of $g$ to be that element.

Hohlweg and Parkinson proposed the following generalization (as we will see in Proposition \ref{vorOG}) of the notion of voracious projection, which we study here.

\begin{defin} \label{vorproj-def}
    Let $B$ be a Garside shadow. The \emph{voracious projection with respect to} $B$ is the map $\vr_B : W \to W$ given by $\nu_B(g) = g \pi_B(g^{-1})$.
\end{defin}

Let $g \in W$ and denote by $p(g)$ the voracious projection defined as above. We verify in the following proposition that the voracious projection $\vr_L(g)$ of $g$ with respect to the set $L$ of low elements is exactly $p(g)$. 

\begin{prop} \label{vorOG}
    $p(g) = \vr_L(g)$ for every $g \in W$.
\end{prop}

\begin{proof}
    We need to show that $\vr_L(g)$ belongs to $P(g)$ and that it is the largest element in that set. Since $\pi_L(g^{-1}) \preceq g^{-1}$, we have that $\vr_L(g) \preceq g$. Furthermore, because $g^{-1}\W(g) \subseteq \Sigma$, we cannot have that $\vr_L(g)$ is separated from $\id$ by any wall in $\W(g)$, otherwise $\pi_L(g^{-1})$ and $g^{-1}$ would be separated by a wall in $\Sigma$, contradicting the fact that they belong to the same Shi part by Lemma \ref{lem-mShi}. Hence, $\vr_L(g) \in P(g)$. Now, suppose there exists an element $h \in P(g)$ such that $\vr_L(g) \preceq h$. Then, $g^{-1}h \preceq \pi_L(g^{-1}) \preceq g^{-1}$. We claim that $g^{-1}h$ is not separated from $g^{-1}$ by any wall in $\Sigma$, so that $g^{-1}h = \pi_L(g^{-1})$ by the definition of $\pi_L$, and thus $h = \vr_L(g)$. Indeed, suppose there is a wall $\W \in \Sigma$ separating $g^{-1}h$ and $g^{-1}$. Then, $g\W$ separates $h$ from $\id$. But note that $g\W \in \W(g)$ since for any wall $\W'$ separating $g\W$ from $g$ we have that $g^{-1}\W'$ separates $\id$ from $\W$, contradicting the fact that $\W$ is elementary. Hence, $h$ is separated from $\id$ by a wall in $\W(g)$ which contradicts the fact that $h \in P(g)$. Therefore, no such wall $\W$ can exist, so $h = \vr_L(g)$.
\end{proof}

Definition \ref{vorproj-def} is thus exactly the same as the original definition of voracious projection, but instead of using the Shi partition $\mathscr{S}$ and its gates $L$, we use the partition $\mathscr{P}_B$ associated to the Garside shadow $B$ and its gates. Furthermore, the map $\vr_B$ also enjoys some of the desirable properties of the original voracious projection. The following lemma is a generalization of \cite[Lem 5.1]{OP-bi}.

\begin{lemma} \label{v-ineq-lemma}
    If $g,g' \in W$ are such that $\vr_B(g) \preceq g' \preceq g$, then $\vr_B(g') \preceq \vr_B(g)$.
\end{lemma}

\begin{proof}
    Let $b \in B$ with $\vr_B(g) = gb$ and let $b' = g'^{-1}\vr_B(g)$. Since $\vr_B(g) \preceq g' \preceq g$, there exists a geodesic $\gamma$ from $\id$ to $g$ in $X^1$ with a subpath from $\vr_B(g)$ to $g'$ labeled by $b'^{-1}$. See Figure \ref{f:ineqfig} for an illustration. Then, translating $\gamma$ by $g^{-1}$ we find that $g^{-1}g' \preceq b$ and that the subpath from $g^{-1}g'$ to $b$ in $g^{-1}\gamma$ is labeled by $b'$. Hence, $b'$ is a suffix of $b$ and thus $b'\in B$. Now, since $g'b' = \vr_B(g)$, translating $\gamma$ by $g'^{-1}$ gives $b' \preceq g'^{-1}$. Thus, since $b' \preceq \pi_B(g'^{-1}) \preceq g'^{-1}$, we obtain $\vr_B(g') = g'\pi_B(g'^{-1}) \preceq g'b' = \vr_B(g)$ as required.
\end{proof}

\begin{figure}[h!]
\begin{center}
\includegraphics[scale=1.1]{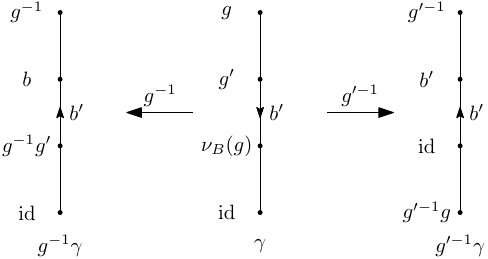}
\end{center}
\caption{Proof of Lemma \ref{v-ineq-lemma}.}
\label{f:ineqfig}
\end{figure}

If the Garside shadow $B$ is finite, we can also guarantee that the voracious projection $\vr_B(g)$ of $g$ does not ``jump too far'', and that the gates $B$ of $\mathscr{P}_B$ are also gates of some $m$-Shi partition.

%\pagebreak

\begin{lemma} \label{constM}
Let $B$ be a finite Garside shadow. Then there exists a constant $M = M(B)$ such that:
\begin{enumerate}
    \item \label{contM-1} $B \subseteq L_M$;
    \item $d(\vr_B(g), g) \leq M$ for every $g \in W$.
\end{enumerate}
\end{lemma}

\begin{proof}
Let $M = \max\{\ell(b) : b \in B\}$. Then, given $b \in B$, any wall separating $b$ from $\id$ is $M$-elementary, so it follows that $b$ is the smallest element in its $M$-Shi part. Hence, $b \in L_M$. For the second part, note that for any $g\in W$ we have that $\vr_B(g) = gb$ for some $b \in B$ and thus $d(\vr_B(g), g) = d(b,\id) = \ell(b) \leq M$.
\end{proof}

We can now also generalize the voracious language of Osajda and Przytycki \cite{OP-bi} by defining it in a similar way for any Garside shadow.

\begin{defin}
    Let $B$ be a Garside shadow. Given $g \in W$, define inductively on $\ell(g)$ a language $\mathcal{L}_g$ of words over $S$ representing $g$ by setting $\LL_{\id} = \{\varepsilon\}$ and
    \begin{align*}
        \LL_g = \{uv : u \in \LL_{\vr_B(g)}, v \in S^* \text{ of minimal length such that } \overline{v} = \vr_B(g)^{-1}g\}.
    \end{align*}
    We define the \emph{voracious language} $\V_B$ of $B$ to be
    \begin{align*}
        \V_B = \bigcup_{g \in W} \LL_g.
    \end{align*}
\end{defin}

Note that $\V_B$ is a \emph{geodesic} language since for each $g \in W$ the language $\LL_g$ consists of minimal length words representing $g$. Furthermore, $\LL_g$ is nonempty and finite for every $g \in W$, so $\V_B$ satisfies condition (\ref{surj L}) of the definition of biautomatic structure. Finally, we note that by Proposition \ref{vorOG}, the original voracious language $\V$ in \cite{OP-bi} is exactly $\V_L$, the voracious language associated to the low elements $L$.

\section{Regularity} \label{section regularity}

We are now ready to define the notion of a regular language and prove that $\V_B$ is regular when $B$ is a finite Garside shadow. 

\begin{defin}
    A \emph{finite state automaton over S} (FSA for short) is a finite directed graph $\Gamma$ together with the following data:
    \begin{enumerate}
        \item a vertex set $A$ and an edge set $E \subseteq A \times A$, where given $a_1,a_2 \in A$ we allow for at most one directed edge from $a_1$ to $a_2$ and at most one directed edge from $a_2$ to $a_1$;
        \item a map $\phi : E \to \mathcal{P}(S^*)$ (the power set of $S^*$) assigning to each edge in $E$ a finite set of labels;
        \item a vertex $a_0 \in A$ called the \emph{start state} of $\Gamma$;
        \item a subset $A_{\infty} \subseteq A$ called the \emph{accept states} of $\Gamma$.
    \end{enumerate}
    The language $\LL(\Gamma)$ of the FSA $\Gamma$ is the set of words $v \in S^*$ which can be decomposed into subwords $v = v_0\ldots v_m$ labelling a directed edge-path $(e_0,\ldots,e_m)$ in $\Gamma$ from $a_0$ to a vertex in $A_{\infty}$, that is $v_i \in \phi(e_i)$ for $i=0,\ldots,m$. A subset $\mathcal{L} \subseteq S^*$ is a \emph{regular language} if it is the language of some FSA $\Gamma$ over $S$.
\end{defin}

Now, we define a FSA $\A_B$ for the language $\V_B$. Since $\A_B$ will be based on $B$ and needs to have a finite number of states, we will require the Garside shadow $B$ to be finite. 

\begin{defin}
    We define the FSA $\A_B$ for $\V_B$ as follows. The vertex set of $\A_B$ is $B$, with start state $\id \in B$ and accept states $B$. The edges of $\A_B$ are defined as follows. Let $b,w \in B$ be such that $\pi_B(wb) = w$. Then, we put an edge $e$ in $E$ from $b$ to $w$ with $\phi(e)$ consisting of all minimal length words representing $w^{-1}$.
\end{defin}

%\hspace{2pt}

\begin{prop}
If $B$ is finite then the language $\V_B$ is regular.
\end{prop}

\begin{proof}
    Let $v \in S^*$ have length $j \geq 0$. We prove by induction on $j$ that:
    \begin{itemize}
        \item $v \in \V_B$ if and only if $v \in \LL(\A_B)$;
        \item the accept state of $v \in \V_B$ is $\pi_B(g^{-1})$ where $\overline{v} = g$.
    \end{itemize}
    For $j=0$ this is true since the empty word is accepted by $\A_B$ with accept state $\id = \pi_B(\id)$. Let $n>0$ and assume the inductive hypothesis is true for all $j<n$. 

    Let $v \in \V_B$ have length $n$ and let $\overline{v} = g \in W$. Let $w = \pi_B(g^{-1})$ and $b = \pi_B((\vr_B(g))^{-1})$. Then, by the inductive hypothesis, $v(\ell(\vr_B(g)))$ is accepted by $\A_B$ with accept state $b$. Now, since $\vr_B(g)b \preceq \vr_B(g) \preceq g$, there is a geodesic $\gamma$ from $\id$ to $g$ in $X^1$ passing through $\vr_B(g)b$ and $\vr_B(g)$. See Figure \ref{f:regfig1} for an illustration. Then, translating $\gamma$ by $g^{-1}$, we find that $w \preceq wb \preceq g^{-1}$ which implies that $\pi_B(wb) = w$ by the definition of $\pi_B$ (since $\pi_B(g^{-1}) = w$). Hence, there is an edge $e$ from $b$ to $w$ in $\A_B$, and since $v(\ell(\vr_B(g))+1, n)$ is a minimal length word representing $w^{-1}$, it belongs to $\phi(e)$. Therefore, $v$ is accepted by $\A_B$.

    \begin{figure}[h!]
    \begin{center}
    \includegraphics[scale=1.1]{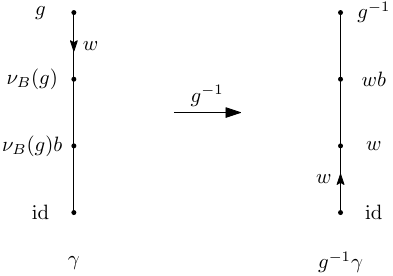}
    \end{center}
    \caption{Proof that $v$ is accepted by $\A_B$.}
    \label{f:regfig1}
    \end{figure}

    Now, let $v \in S^*$ of length $n$ be accepted by $\A_B$. Let $v=v_0\ldots v_m$ be a decomposition of $v$ into subwords respecting the definition of $\A_B$. By the inductive hypothesis, $v_0\ldots v_{m-1} \in \V_B$ and represents some element $g \in W$. Furthermore, the last edge $e_m$ labeled by $v_m$ starts at $b = \pi_B(g^{-1})$ and ends in the element $w$ of $B$ that is the inverse of $\overline{v}_m$. In particular, since $v$ is accepted by $\A_B$, we have that $\pi_B(wb) = w$. We will show that $\vr_B(gw^{-1}) = g$ thus implying that $v \in \V_B$. We first claim that $w^{-1} \in T_{g^{-1}, \id}$, that is that $w^{-1}$ belongs to the cone type of $g$. Since $\pi_B(wb) = w$, we have that $w \preceq wb$. Hence, translating by $w^{-1}$ a geodesic from $\id$ to $wb$ through $w$, we find that $w^{-1} \in T_{b,\id}$. Let $b' = \pi_{\Gamma}(g^{-1})$, that is the gate of the cone type part containing $g^{-1}$ by Lemma \ref{lem-cone}. Then, since $\Gamma \subseteq B$ by Proposition \ref{cone type garside}, we have that $b' \preceq b \preceq g^{-1}$ and thus $T_{g^{-1}, \id} = T_{b', \id} = T_{b, \id}$ by Lemma \ref{lem-convex}. This justifies the claim. Hence, there exists a geodesic $\gamma$ in $X^1$ from $g^{-1}$ to $w^{-1}$ passing through $b$ and $\id$. Then, translating $\gamma$ by $w$ we find that $w \preceq wb \preceq wg^{-1}$. Suppose for a contradiction that $\vr_B(gw^{-1}) \neq g$. Then, there exists $b' \in B$ such that $gw^{-1}b' \preceq g \preceq gw^{-1}$ and thus by Lemma \ref{v-ineq-lemma} we get $gb \preceq gw^{-1}b' \preceq g \preceq gw^{-1}$. But then, translating by $(gw^{-1})^{-1}$ we obtain $w \preceq b' \preceq wb $ (see Figure \ref{f:regfig2}), which contradicts the fact that $\pi_B(wb) = w$. Hence, we have that $\vr_B(gw^{-1}) = g$ and so $v \in \V_B$ with accept state $w = \pi_B((gw^{-1})^{-1})$.
\end{proof}

\begin{figure}[h!]
    \begin{center}
    \includegraphics[scale=1.1]{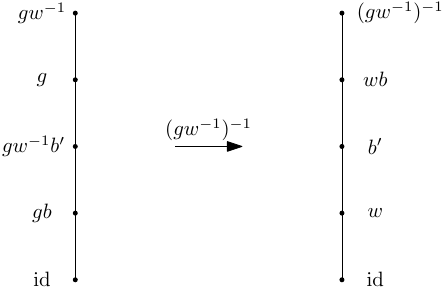}
    \end{center}
    \caption{Proof that $v$ belongs to $\V_B$.}
    \label{f:regfig2}
    \end{figure}

\begin{rem}
    Let $v \in S^*$ and $v = v_0\ldots v_m$ be a decomposition of $v$ into subwords where $g_i = \overline{v_i}$. If $g_0^{-1} \in B$ and $\vr_B(g_ig_{i+1}) = g_i$ for each $i = 0, \ldots, m-1$, then $v \in \V_B$. Indeed, the second condition gives us that $\pi_B(g_{i+1}^{-1}g_i^{-1}) = g_{i+1}^{-1} \in B$, so there is an edge $e$ in $\A_B$ from $g_i^{-1}$ to $g_{i+1}^{-1}$ with $v_{i+1} \in \phi(e)$. Hence, each normal form $v_0\ldots v_m$ can be recognized locally by checking that $g_0^{-1} \in B$ and that $v_iv_{i+1}$ is a normal form for each $i = 0, \ldots, m-1$.
\end{rem}

\section{Fellow traveller property} \label{section FTP}

In this section, we prove the first and second fellow traveller property for $\V_B$ when $B$ is a finite Garside shadow.

\begin{prop} \label{PropFTPa}
    If $B$ is finite then $\V_B$ satisfies the first fellow traveller property with $C = 2M$, where $M$ is the constant from Lemma \ref{constM}.
\end{prop}

\begin{proof}
    The proof is essentially the same as to that of \cite[Cor~5.2]{OP-bi}. Let $g \in W$, $s \in S$ and $g' = gs$. Assume without loss of generality that $\ell(g') < \ell(g)$. Then, $s \preceq g^{-1}$ and so $s \preceq \pi_B(g^{-1}) \preceq g^{-1}$ by definition of $\pi_B$ since $s \in B$. Thus, we find that $\vr_B(g) \preceq g' \preceq g$ and we can iteratively apply Lemma \ref{v-ineq-lemma} to obtain
    \begin{align*}
        \id = \vr_B^n(g') \preceq \vr_B^n(g) \preceq \ldots \preceq \vr_B^2(g') \preceq \vr_B^2(g) \preceq \vr_B(g') \preceq \vr_B(g) \preceq g' \preceq g 
    \end{align*}
    for some $n \in \N$, where $\vr_B^k$ denotes the $k$-fold composition of $\vr_B$ (for convenience also define $\vr_B^0(g) = g$). Let $v,v' \in \V_B$ be words representing $g$ and $g'$ respectively. If $\ell(\vr_B^k(g')) \leq i \leq \ell(\vr_B^k(g))$ for some $k \geq 0$, then 
    \begin{align*}
        d(\overline{v(i)},\overline{v'(i)}) \leq d(\overline{v(i)},\vr_B^{k+1}(g)) + d(\vr_B^{k+1}(g),\overline{v'(i)}) \leq 2M
    \end{align*}
    by Lemma \ref{constM}. Otherwise, we have $\ell(\vr_B^{k+1}(g)) \leq i \leq \ell(\vr_B^k(g'))$ for some $k \geq 0$ and thus
    \begin{align*}
        d(\overline{v(i)},\overline{v'(i)}) \leq d(\overline{v(i)},\vr_B^{k+1}(g')) + d(\vr_B^{k+1}(g'),\overline{v'(i)}) \leq 2M
    \end{align*}
    by Lemma \ref{constM} again.
\end{proof} 

To prove the second fellow traveller property we will need the following result, which is a generalization of the  Parallel Wall Theorem proven by Brink and Howlett \cite[Thm 2.8]{BH-par}, and is a direct consequence of the fact that $\Sigma_m$ is finite for every $m \in \N$  \cite{Fu-E_m}.

\begin{thm} \label{Parallel wall}
    Let $W$ be a Coxeter group and $m \in \N$. There exists a constant $Q_m = Q(W,m)$ such that for each $g \in W$ and wall $\W$ at distance greater than $Q_m$ from $g$ in $X^1$, there are $m$ walls separating $g$ from $\W$.
\end{thm}

We can now finish the proof of Theorem \ref{mainthm} with the following.

\begin{prop} \label{propFTPb}
    If $B$ is finite then $\V_B$ satisfies the second fellow traveller property with $C' = 4M(M+Q_M)+2Q_M$, where $M$ is the constant from Lemma \ref{constM} and $Q_M$ is the constant from Theorem \ref{Parallel wall}.
\end{prop}

\begin{proof}
    Let $g \in W$ and $s \in S$. We will prove this by induction on $\ell(g)$. Assume without loss of generality that $\ell(sg)>\ell(g)$. For $g = \id$ there is nothing to prove, so suppose $g \neq \id$. Let $v,v' \in \V_B$ be words representing $g$ and $sg$ respectively. 
    
    Suppose first that $\pi_B(g^{-1}) = \pi_B((sg)^{-1})$. Then,
    \begin{align*}
        s\vr_B(g) = sg\pi_B(g^{-1}) = sg\pi_B((sg)^{-1}) = \vr_B(sg)
    \end{align*}
    so $v'(\ell(\vr_B(sg)))$ and $sv(\ell(\vr_B(g)))$ represent the same element of $W$. Hence, the second fellow traveller property is satisfied by induction for $i < \ell(\vr_B(sg))$ and by Lemma \ref{constM} for $i \geq \ell(\vr_B(sg))$.

    Now, assume that $\pi_B(g^{-1}) \neq \pi_B((sg)^{-1})$. Consider the wall $\W_s$ separating $\id$ from $s$. Since $\vr_B(g) \preceq g$ we have that $s \preceq s\vr_B(g) \preceq sg$, so $s\vr_B(g)$ lies on the same side of $\W_s$ as $s$. If this is also the case for $\vr_B(sg)$, then $s \preceq \vr_B(sg)$ since we can reflect through $\W_s$ the part of the geodesic from $\id$ to $\vr_B(sg)$ in $X^1$ lying on the same side of $\W_s$ as $\id$. We thus have a geodesic $\gamma$ from $\id$ to $sg$ in $X^1$ passing through $s$ and $\vr_B(sg)$, but then translating $\gamma$ by $(sg)^{-1}$ we find that $\pi_B((sg)^{-1}) \preceq g^{-1} \preceq (sg)^{-1}$, which implies that $\pi_B(g^{-1}) = \pi_B((sg)^{-1})$ by definition of $\pi_B$, giving a contradiction. Hence, $\vr_B(sg)$ lies on the same side of $\W_s$ as $\id$. See Figure \ref{f:ftpfig} for an illustration.

    \begin{figure}[h]
    \begin{center}
    \includegraphics[scale=1.1]{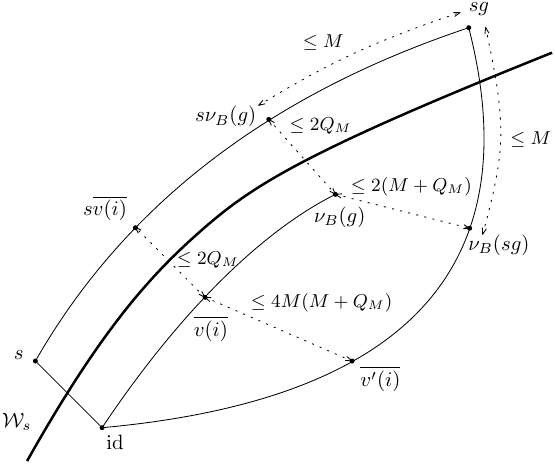}
    \end{center}
    \caption{Proof of Proposition \ref{propFTPb}.}
    \label{f:ftpfig}
    \end{figure}
    
    Finally, we claim that the wall $\W_s$ is $M$-close to $sg$. Assume for a contradiction that it is not. Then, $(sg)^{-1}\W_s$ is not $M$-elementary, and since it separates $g^{-1}$ from $(sg)^{-1}$, we must have that $\pi_{L_M}(g^{-1}) = \pi_{L_M}((sg)^{-1})$ as a consequence of Lemma \ref{lem-mShi}. But since $B \subseteq L_M$ by Lemma \ref{constM}, we have by Lemma \ref{refinement} that $\pi_B(g^{-1}) = \pi_B((sg)^{-1})$ which is a contradiction. Hence, $\W_s$ is $M$-close to $sg$. 

    Let $1 \leq i \leq \ell(v)$ and $\overline{v(i)} = h \in W$ so that $s \preceq sh \preceq sg$. Then, $\W_s$ is also $M$-close to $sh$ since any wall separating $\W_s$ from $sh$ also separates $\W_s$ from $sg$ (otherwise it intersects the geodesic $sv$ twice). Hence, by Theorem \ref{Parallel wall} we have that $d(sh, h) \leq 2Q_M$ and in particular $d(s\vr_B(g), \vr_B(g)) \leq 2Q_M$. Therefore,
    \begin{align*}
        d(sg, \vr_B(g)) \leq d(sg, s\vr_B(g)) + d(s\vr_B(g), \vr_B(g)) \leq M + 2Q_M
    \end{align*}
    and thus $d(\vr_B(g), \vr_B(sg)) \leq 2(M+Q_M)$. Hence, by applying Proposition \ref{PropFTPa} repeatedly we have for $i < \ell(\vr_B(sg))$ that $d(\overline{v(i)}, \overline{v'(i)}) \leq 4M(M+Q_M)$ and thus
    \begin{align*}
        d(s\overline{v(i)}, \overline{v'(i)}) \leq 4M(M+Q_M) + 2Q_M.
    \end{align*}
    Finally, since $d(s\vr_B(g), \vr_B(sg)) \leq 2M$, we have that $d(s\overline{v(i)}, \overline{v'(i)}) \leq 2M$ for all $i \geq \ell(\vr_B(sg))$.
\end{proof}

% \medskip

% \noindent \textbf{Data availability.} The author declares that his manuscript has no associated data.

% \medskip

% \noindent \textbf{Conflict of interest.} The author declares that he has no conflict of interest.

\bibliography{refs}{}

% \bib, bibdiv, biblist are defined by the amsrefs package.
\begin{bibdiv}
\begin{biblist}

\bib{biautomaticdef}{misc}{
      author={Amrhein, Aischa},
       title={Characterizing biautomatic groups},
   publisher={arXiv:2105.07509},
        date={2021},
         url={https://arxiv.org/abs/2105.07509},
}

\bib{Ba-bi}{article}{
      author={Bahls, Patrick},
       title={Some new biautomatic {C}oxeter groups},
        date={2006},
        ISSN={0021-8693,1090-266X},
     journal={J. Algebra},
      volume={296},
      number={2},
       pages={339\ndash 347},
         url={https://doi.org/10.1016/j.jalgebra.2005.12.003},
}

\bib{B-meet}{book}{
      author={Bj\"orner, Anders},
      author={Brenti, Francesco},
       title={Combinatorics of {C}oxeter groups},
      series={Graduate Texts in Mathematics},
   publisher={Springer, New York},
        date={2005},
      volume={231},
        ISBN={978-3540-442387; 3-540-44238-3},
}

\bib{BH-par}{article}{
      author={Brink, Brigitte},
      author={Howlett, Robert~B.},
       title={A finiteness property and an automatic structure for {C}oxeter groups},
        date={1993},
        ISSN={0025-5831,1432-1807},
     journal={Math. Ann.},
      volume={296},
      number={1},
       pages={179\ndash 190},
         url={https://doi.org/10.1007/BF01445101},
}

\bib{Ca-bi}{article}{
      author={Caprace, Pierre-Emmanuel},
       title={Buildings with isolated subspaces and relatively hyperbolic {C}oxeter groups},
        date={2009},
        ISSN={2640-7337,2640-7345},
     journal={Innov. Incidence Geom.},
      volume={10},
       pages={15\ndash 31},
}

\bib{CM-bi}{article}{
      author={Caprace, Pierre-Emmanuel},
      author={M\"uhlherr, Bernhard},
       title={Reflection triangles in {C}oxeter groups and biautomaticity},
        date={2005},
        ISSN={1433-5883,1435-4446},
     journal={J. Group Theory},
      volume={8},
      number={4},
       pages={467\ndash 489},
         url={https://doi.org/10.1515/jgth.2005.8.4.467},
}

\bib{Coxeter-1934}{article}{
      author={Coxeter, H. S.~M.},
       title={Discrete groups generated by reflections},
        date={1934},
        ISSN={0003-486X,1939-8980},
     journal={Ann. of Math. (2)},
      volume={35},
      number={3},
       pages={588\ndash 621},
         url={https://doi.org/10.2307/1968753},
}

\bib{DS-au}{article}{
      author={Davis, Michael~W.},
      author={Shapiro, Michael~D.},
       title={Coxeter groups are automatic},
        date={1991},
       pages={1\ndash 16},
        note={Ohio State Mathematical Research Institute Preprints, no. 91-15},
}

\bib{D-monoid}{book}{
      author={Dehornoy, Patrick},
      author={Digne, François},
      author={Godelle, Eddy},
      author={Krammer, Daan},
      author={Michel, Jean},
       title={Foundations of {G}arside theory},
      series={EMS Tracts in Mathematics},
   publisher={European Mathematical Society (EMS), Z\"urich},
        date={2015},
      volume={22},
        ISBN={978-3-03719-139-2},
         url={https://doi.org/10.4171/139},
}

\bib{DDH-monoid}{article}{
      author={Dehornoy, Patrick},
      author={Dyer, Matthew},
      author={Hohlweg, Christophe},
       title={Garside families in {A}rtin-{T}its monoids and low elements in {C}oxeter groups},
        date={2015},
        ISSN={1631-073X,1778-3569},
     journal={C. R. Math. Acad. Sci. Paris},
      volume={353},
      number={5},
       pages={403\ndash 408},
         url={https://doi.org/10.1016/j.crma.2015.01.008},
}

\bib{D-mlow}{article}{
      author={Dyer, M.~J.},
       title={{$n$}-low elements and maximal rank {$k$} reflection subgroups of {C}oxeter groups},
        date={2022},
        ISSN={0021-8693,1090-266X},
     journal={J. Algebra},
      volume={607},
       pages={139\ndash 180},
         url={https://doi.org/10.1016/j.jalgebra.2021.02.015},
}

\bib{DHFM24}{article}{
      author={Dyer, Matthew},
      author={Fishel, Susanna},
      author={Hohlweg, Christophe},
      author={Mark, Alice},
       title={Shi arrangements and low elements in {C}oxeter groups},
        date={2024},
        ISSN={0024-6115,1460-244X},
     journal={Proc. Lond. Math. Soc. (3)},
      volume={129},
      number={2},
       pages={Paper No. e12624, 48},
         url={https://doi.org/10.1112/plms.12624},
}

\bib{DH-Garside}{article}{
      author={Dyer, Matthew},
      author={Hohlweg, Christophe},
       title={Small roots, low elements, and the weak order in {C}oxeter groups},
        date={2016},
        ISSN={0001-8708,1090-2082},
     journal={Adv. Math.},
      volume={301},
       pages={739\ndash 784},
         url={https://doi.org/10.1016/j.aim.2016.06.022},
}

\bib{wordprocessing}{book}{
      author={Epstein, David B.~A.},
      author={Cannon, James~W.},
      author={Holt, Derek~F.},
      author={Levy, Silvio V.~F.},
      author={Paterson, Michael~S.},
      author={Thurston, William~P.},
       title={Word processing in groups},
   publisher={Jones and Bartlett Publishers, Boston, MA},
        date={1992},
        ISBN={0-86720-244-0},
}

\bib{Fu-E_m}{article}{
      author={Fu, Xiang},
       title={The dominance hierarchy in root systems of {C}oxeter groups},
        date={2012},
        ISSN={0021-8693,1090-266X},
     journal={J. Algebra},
      volume={366},
       pages={187\ndash 204},
         url={https://doi.org/10.1016/j.jalgebra.2012.05.013},
}

\bib{HNW-Gar}{article}{
      author={Hohlweg, Christophe},
      author={Nadeau, Philippe},
      author={Williams, Nathan},
       title={Automata, reduced words and {G}arside shadows in {C}oxeter groups},
        date={2016},
        ISSN={0021-8693,1090-266X},
     journal={J. Algebra},
      volume={457},
       pages={431\ndash 456},
         url={https://doi.org/10.1016/j.jalgebra.2016.04.006},
}

\bib{MOP-bi}{article}{
      author={Munro, Zachary},
      author={Osajda, Damian},
      author={Przytycki, Piotr},
       title={2-dimensional {C}oxeter groups are biautomatic},
        date={2022},
        ISSN={0308-2105,1473-7124},
     journal={Proc. Roy. Soc. Edinburgh Sect. A},
      volume={152},
      number={2},
       pages={382\ndash 401},
         url={https://doi.org/10.1017/prm.2021.11},
}

\bib{NR-bi}{article}{
      author={Niblo, G.~A.},
      author={Reeves, L.~D.},
       title={The geometry of cube complexes and the complexity of their fundamental groups},
        date={1998},
        ISSN={0040-9383},
     journal={Topology},
      volume={37},
      number={3},
       pages={621\ndash 633},
         url={https://doi.org/10.1016/S0040-9383(97)00018-9},
}

\bib{NR-cox}{article}{
      author={Niblo, G.~A.},
      author={Reeves, L.~D.},
       title={Coxeter groups act on {${\rm CAT}(0)$} cube complexes},
        date={2003},
        ISSN={1433-5883,1435-4446},
     journal={J. Group Theory},
      volume={6},
      number={3},
       pages={399\ndash 413},
         url={https://doi.org/10.1515/jgth.2003.028},
}

\bib{OP-bi}{article}{
      author={Osajda, Damian},
      author={Przytycki, Piotr},
       title={Coxeter groups are biautomatic},
        date={2025},
        ISSN={0020-9910,1432-1297},
     journal={Invent. Math.},
      volume={242},
      number={3},
       pages={627\ndash 637},
         url={https://doi-org.proxy3.library.mcgill.ca/10.1007/s00222-025-01365-6},
}

\bib{PaY-reg}{article}{
      author={Parkinson, James},
      author={Yau, Yeeka},
       title={Cone types, automata, and regular partitions in {C}oxeter groups},
        date={2022},
        ISSN={0001-8708,1090-2082},
     journal={Adv. Math.},
      volume={398},
       pages={Paper No. 108146, 66},
         url={https://doi.org/10.1016/j.aim.2021.108146},
}

\bib{PrY-pair}{article}{
      author={Przytycki, Piotr},
      author={Yau, Yeeka},
       title={A pair of {G}arside shadows},
        date={2024},
        ISSN={2589-5486},
     journal={Algebr. Comb.},
      volume={7},
      number={6},
       pages={1879\ndash 1885},
         url={https://doi.org/10.5802/alco.387},
}

\bib{R-buildings}{book}{
      author={Ronan, Mark},
       title={Lectures on buildings},
   publisher={University of Chicago Press, Chicago, IL},
        date={2009},
        ISBN={978-0-226-72499-7; 0-226-72499-9},
        note={Updated and revised},
}

\end{biblist}
\end{bibdiv}
\bibliographystyle{plain}

\end{document}